\documentclass{article}

\usepackage{arxiv}

\usepackage[utf8]{inputenc} % allow utf-8 input
\usepackage[T1]{fontenc}    % use 8-bit T1 fonts
\usepackage{hyperref}       % hyperlinks
\usepackage{url}            % simple URL typesetting
\usepackage{booktabs}       % professional-quality tables
\usepackage{amsfonts}       % blackboard math symbols
\usepackage{nicefrac}       % compact symbols for 1/2, etc.
\usepackage{microtype}      % microtypography
\usepackage{lipsum}
\usepackage{amsmath}

\usepackage{color}

\usepackage{graphicx}

\usepackage{subfig}
\graphicspath{{figuras/}{fig_site/}}

\newtheorem{definition}{\noindent{\it Definition}}[section]
\newtheorem{theorem}{\noindent{\it Theorem}}[section]

\newtheorem{remark}[theorem]{\noindent{\it Remark}}
\newtheorem{corollary}[theorem]{\noindent{\it Corollary}}
\newenvironment{proof}{\noindent{\it Proof:}}{$\hfill$ $\Box$\\ }

\title{A numerical method for a class of nonlinear fractional advection-diffusion equations}

\author{
Jocemar de Q. Chagas%\thanks{Corresponding author.} 
\\
Departamento de Matemática e Estatística\\
Universidade Estadual de Ponta Grossa\\
Ponta Grossa, PR 84030-900, Brazil \\
\texttt{jocemarchagas@uepg.br} \\
\And
Giuliano G. La Guardia \\
Departamento de Matemática e Estatística\\
Universidade Estadual de Ponta Grossa\\
Ponta Grossa, PR 84030-900, Brazil \\
\texttt{gguardia@uepg.br} \\
\And
Ervin K. Lenzi \\
Departamento de Física\\
Universidade Estadual de Ponta Grossa\\
Ponta Grossa, PR 84030-900, Brazil \\
\texttt{eklenzi@uepg.br} \\
}

\begin{document}
\maketitle
	
\begin{abstract}
	In this note, a numerical method based on finite differences to solve a class of nonlinear advection-diffusion fractional differential equation is proposed. The fractional operator considered here is the fractional Riemann-Liouville derivative or the fractional Riesz derivative of order $\alpha$. The consistency and unconditionally stability of the method are shown. Finally, an example of application of this method is presented.  
\end{abstract}
	

\section{Introduction}
	
A numerical approach based on finite difference method is proposed for solving the nonlinear fractional diffusion equation:
\begin{eqnarray} \label{Eq01}
	\frac{\partial }{\partial t}{u}(x,t) =  {c}^{2}\frac{\partial^{\alpha}}{\partial {|x|}^{\alpha}}{u}^{\nu}(x,t) ,
\end{eqnarray}
where $L \leq x \leq R$ and $t > 0$.
%, and the derivative operator   $\frac{\partial^{\alpha}}{\partial {|x|}^{\alpha}}$ is in the sense of Riemann-Liouville or of Riesz. 
	
Here, $u(x,t) \geq 0$ is a unknown function, which may be related, e.g, to an density of probability, $c^2$ represents the diffusion coefficient, and the fractional operator $\displaystyle \frac{\partial^{\alpha}}{\partial {|x|}^{\alpha}}$ is the fractional Riemann-Liouville derivative or the fractional Riesz derivative of order $\alpha$ (for details, see Ref.~\cite{Li2018}). It is worth mentioning that different fractional operators (for instance, see Refs.~\cite{caputo2015new,atangana2016new,YANG2017276}) have been utilized to investigate several situations, such as neumatic liquid crystal~\cite{neumaticliquidcrystal}, $(2 + 1)$ - dimensional mKdV equation~\cite{hosseini2020detailed}, optical solitons~\cite{HOSSEINI2020164801}, anomalous diffusion~\cite{10.3389/fphy.2017.00052}, Chua's circuit model~\cite{ALKAHTANI2016547}, among others. We also consider $1 < \alpha \leq 2$ and $\nu\neq 1$ which are usually related to unusual relaxation processes and allow us to deal with problems related to anomalous diffusion process  (see for instance \cite{Book2018}). Further, it is worth mentioning that Eq.~(\ref{Eq01}) may be related to a nonlinear stochastic equation with a colored noise~\cite{10.1063}. 
We present an implicit Euler method, which holds for $\nu > 0$, and show that it is consistent and unconditionally stable, therefore, convergent from the Rosinger Theorem~\cite{Rosinger}, which is a nonlinear extension of the celebrated Lax-Richtmyer equivalence theorem~\cite{Lax}. Additionally, we indicate a form of iterating that allows us to implement the method. 
	
We observe that our method presented here can be directly extended for a more complete set of equations, i.e., nonlinear advection-diffusion fractional equation, of the form
\begin{eqnarray} \label{Eq02}
	\frac{\partial}{\partial t}{u}(x,t) = - a(x) \frac{\partial }{\partial x} {u}(x,t) +
	b(x) \frac{\partial^{\alpha}}{\partial {|x|}^{\alpha}}{u}^{\nu}(x,t)+f(x,t)\; ,
\end{eqnarray}
where $a(x), b(x) \geq 0$ for all $x \in \mathbb{R}$. For simplicity, we only present the proof for Eq.~(\ref{Eq01}).

\section{Preliminaries}
	
In this section, we present some concepts used in this note. According to Ref.~\cite{Capelas}, we have the following definitions.
	
\begin{definition} \label{L}
The Liouville fractional derivative  with order $\alpha > 0$ of a given function $f(x)$, $x \in \mathbb{R}$, is defined as
\begin{eqnarray} \label{Eq03e}
	{}^{}{D}^{\alpha} f(x) = 
	\frac{1}{\Gamma(1-\alpha)} \frac{{d}}{d {x}} \int_{-\infty}^{x} \frac{f(\eta)}{{(x - \eta)}^{\alpha}} d\eta \, ,
\end{eqnarray}
where $\Gamma(\cdot)$ is the Euler's gamma function.
\end{definition}
	
\begin{definition} \label{LRRL}
The left and right Riemann-Liouville fractional derivative  with order $\alpha > 0$ of a given function $f(x)$, $x \in (a,b)$, are defined, respectively, as
\begin{eqnarray} \label{Eq03}
	{}^{RL}{D}_{a,x}^{\alpha} f(x) = 
	\frac{1}{\Gamma(n-\alpha)} \frac{{d}^{n}}{d {x}^{n}} \int_{a}^{x} \frac{f(\eta)}{{(x - \eta)}^{\alpha + 1 - n}} d\eta
\end{eqnarray}
\begin{eqnarray*}\label{Eq03a}
	{}^{RL}{D}_{x,b}^{\alpha} f(x) = 
	\frac{{(-1)}^{n}}{\Gamma(n-\alpha)} \frac{{d}^{n}}{d {x}^{n}} \int_{x}^{b} \frac{f(\eta)}{{(x - \eta)}^{\alpha + 1 - n}} d\eta, 
\end{eqnarray*}
where $n$ is an positive integer such that $n-1 < \alpha \leq n$. As usual in the literature, we call by Riemann-Liouville fractional derivative the left fractional derivative in (\ref{Eq03}).
\end{definition}
	
\begin{definition} \label{RL}
The 
%left and right 
Grünwald-Letnikov derivative of order $\alpha > 0$ of a given function $f(x)$, $x \in (a,b)$, are defined, respectively, as
\begin{eqnarray*} \label{Eq03b}
	{}^{GL}{D}_{a,x}^{\alpha} f(x) = \lim_{h \to 0}
	\frac{1}{h^\alpha} \sum_{k = 0}^{\lfloor n \rfloor} {(-1)}^{k}
	\frac{\Gamma(\alpha+1)f(x-kh)}{\Gamma(k+1)\Gamma(\alpha-k+1)},  \quad nh = x - a.
	%  \\
	%  \label{Eq03c}
	%  {}^{GL}{D}_{x,b}^{\alpha} f(x) = \lim_{h \to 0}
	%  \frac{1}{h^\alpha} \sum_{k = 0}^{\lfloor n \rfloor} {(-1)}^{n}
	%  \frac{\Gamma(\alpha+1)f(x+kh)}{\Gamma(k+1)\Gamma(\alpha-k+1)}, \quad nh = b - x.
\end{eqnarray*}
\end{definition}
	
\begin{definition} \label{RZ}
The Riesz fractional derivative of order $\alpha$ is defined as
\begin{eqnarray*} \label{Eq03.1}
	{}^{RZ}{D}_{x}^{\alpha} f(x) = \frac{-1}{2cos(\alpha \pi /2)} \left( {}^{RL}{D}_{a,x}^{\alpha} f(x) + {}^{RL}{D}_{x,b}^{\alpha} f(x) \right) \,,
\end{eqnarray*}
if $\alpha \neq 2k+1$, $k = 0, 1, 2, \cdots$. 
\end{definition}
	
The following definitions are necessary for the approximations used in the discretizations.
	
\begin{definition} \label{BD}
The standard backward difference operator and the second-order centered difference operator are given respectively by 
\begin{eqnarray*} \label{Eq03t}
	\displaystyle 
	{\mathbb D}_{+} ({v}_{m}) &=& \frac{{v}_{m}-{v}_{m-1}}{h}, \\
	\label{Eq03t1}
	\displaystyle 
	{\mathbb D}_{+}{\mathbb D}_{-} ({v}_{m}) &=& \frac{{v}_{m+1}+2v_{m}-{v}_{m-1}}{h^{2}}.
\end{eqnarray*}
\end{definition}
	
\begin{definition} \label{GL} 
The standard Grünwald-Letnikov approximation (\cite{podlubny}, \cite{Li}, \cite{Karniadakis}) is given by :  
\begin{eqnarray*} \label{Eq04}
	{}^{GL}{D}_{a,x}^{\alpha} f(x)
	& \approx & 
	h^{- \alpha} \sum_{i = 0}^{M} w_{i}^{\alpha} f(x - ih) \, ,
\end{eqnarray*}
where $\displaystyle w_{i}^{\alpha} = {(-1)}^{i} {{\alpha}\choose{i}}$.
\end{definition}
	
\begin{definition} \label{ShiG} 
The shifted Grünwald formula (\cite{Li2018}, \cite{Karniadakis}, \cite{meerschaert}) is given by:
\begin{eqnarray} \label{Eq06}
	{}^{sGL}{D}_{a,x}^{\alpha} f(x)
	& \approx & 
	h^{-{\alpha}} \sum_{i = 0}^{M+1} w_{i}^{\alpha}f(x - (i-p)h),
\end{eqnarray}
where $p$ is a non-negative integer.
\end{definition}
	
We next recall the well-known definition, which will be useful in our analysis:
	
\begin{definition}\label{def1}
Let $A \in {\mathbb{R}}^{n \times n}$ be a matrix. We say that $A$ is strictly row diagonally dominant if
\begin{eqnarray*} \label{Eq07}
	| a_{ii}| > \displaystyle\sum_{j=1, j\neq i}^{n}|a_{ij}|,
\end{eqnarray*} for $i \in \{1,2,\cdots,n\}$.
\end{definition}
	
Based on Definition~\ref{def1}, it follows Theorem 4.1.2 of \cite{Golub}:
	
\begin{theorem} \label{Thm1}
If $A \in {\mathbb{R}}^{n \times n}$ is strictly row diagonally dominant then $A$ is invertible and ${\| A^{-1}\|}_{1} \leq 1/\delta$, where 
\begin{eqnarray*} %\label{Eq08}
	\delta= \displaystyle\min_{1 \leq j \leq n} \left(|a_{jj}|-\displaystyle\sum_{i=1, i\neq j}^{n}|a_{ij}|\right) > 0.
\end{eqnarray*}
\end{theorem}
	
We next recall an important result due to Rosinger~\cite{Rosinger} about the equivalence on convergence and stability of systems of nonlinear evolution equations:
	
Let $(X,\|\cdot\|)$ be a normed vector space  and let us consider the nonlinear evolution equation
\begin{equation}\label{ProblemNonLinear} 
	\left\{
	\begin{array}{c} 
		\displaystyle 
		\frac{d}{dt}u(t) = A\bigl(u(t)\bigr), \ t \in [0,T] \\
		u(0) = f
	\end{array} 
	\right.
\end{equation}
where $A: X_1 \to X$, $X_1 \subset X$ is a set, and $u(t), f \in X_1$. The elements of $X$ are functions of a variable $x \in {\mathbb R}^{m}$, with $f = f(x)$ and $u(t) = F(x,t)$.
	
\begin{theorem}(Rosinger~\cite{Rosinger}) \label{ThmR}
Suppose that the problem (\ref{ProblemNonLinear}) is properly posed and we assume that a difference scheme which is consistent with the problem is proposed. Then the difference scheme is convergent to the solution of the problem (\ref{ProblemNonLinear}) if and only if it is stable. 
\end{theorem}

\section{Numerical Method}
	
After reviewing some concepts in the previous section, we here present a numerical approach to obtain solutions to the nonlinear diffusion equation (\ref{Eq01}), based on finite differences. The Finite Difference Method is one of the techniques for searching numerical solutions for partial differential equations (PDEs) as well as fractional partial differential equations (FPDEs). In the case of integer order, the theory is well-known (see, for instance, \cite{Strikwerda}) and it is based on a suitable discretization of the equations on a grid of points (or mesh) in the domain. In this note, for simplicity, we deal with a one-dimensional case. For linear equations with derivatives of fractional order, some methods were available (see \cite{podlubny}, \cite{Li}, \cite{Karniadakis}; see also the review paper  \cite{Li2018}). Here, we adapt to the nonlinear case the implicit Euler method with shifted Grünwald formula given by Theorem 2.7 in Ref.~\cite{meerschaert}.
	
In the following, we consider that $k$ and $h$ are positive real numbers, called {\it time-step} and {\it space-step}, respectively. The exact solution $u(x,t)$, when evaluated in the grid point $(x_j,t^n)$, is denoted by $u(x_j,t^n)$; more briefly ${u}_{j}^{n}$. We consider a mesh with $0 \leq j \leq M$ and  $0 \leq n \leq N$. The boundary values of the domain are $x_0 = L$ and $x_M = R$, and $t^N = t$ denotes the final time. When numerical solutions are considered, we utilize an approximation for the exact solution $u$ in the grid points, according to the notation $v(x_j,t^n)$, or ${v}_{j}^{n}$. When we consider Eq.~(\ref{Eq01}) in the grid points we look at the system of discretized equations:
\begin{eqnarray} \label{Eq09}
	\frac{\partial }{\partial t}v(x_j,t^n) 
	& = & 
	c^2 \frac{\partial^{\alpha} }{\partial {|x|}^{\alpha}} {v}^{\nu}(x_j,t^n) \; .
\end{eqnarray}
The problem is to find a suitable discretization for each term of Eq.~(\ref{Eq09}) in order to assure the convergence of the numerical method proposed here. As it is known, it is not easy to find a simple approach to address this problem. In this light, our main contribution is to propose a simple and reliable numerical method to solve nonlinear fractional partial equations, which are represented by Eq.~(\ref{Eq09}).
	
We start by assuming that $\nu > 0$ and $\delta > -1$ such as $\nu = 1 + \delta$. After this, we multiply the Eq.~(\ref{Eq01}) by $(1+\delta) {u}^{\delta}(x,t)$:
\begin{eqnarray*} \label{Eq2.10}
	(1+\delta) {u}^{\delta}(x,t)  \frac{\partial }{\partial t} {u}(x,t)
	& = &
	(1+\delta) {u}^{\delta}(x,t)  c^2
	\frac{\partial^{\alpha} }{\partial {|x|}^{\alpha}}{u}^{\nu}(x,t) \, ,
\end{eqnarray*}
yielding
\begin{eqnarray} \label{Eq11}
	\frac{\partial }{\partial t} {u}^{\nu}(x,t)
	&=&
	\nu \, {u}^{\delta}(x,t) \, c^2 \frac{\partial^{\alpha} }{\partial {|x|}^{\alpha}} {u}^{\nu}(x,t) \, .
\end{eqnarray}
	
%\textcolor{blue}{Note that the functions on which the derivatives act in the nonlinear equation (\ref{Eq11}) are equal, hence the argument for the linear case in Theorem 2.7 of \cite{meerschaert} has a possibility to be adapted to our case, if in each step of time, we can properly evaluate the weights $\nu \, {u}^{\delta}(x,t)$. }
	
In the grid points, we denote the weights $\nu \, {v}^{\delta}(x_j,t^n)$ only by $\delta_j^n := \nu \, {({v}^{\delta})}_j^n = \nu \, {v}^{\delta}(x_j,t^n)$. 
	
In the following, as our main result, we propose a convergent implicit Euler method to solve Eq.~(\ref{Eq09}), and an approach to evaluate the weights $\delta_j^n$.
	
\begin{theorem}\label{Thm2}
The implicit Euler method
\begin{eqnarray} \label{Eq12}
	-\lambda {\delta}_{j}^{n+1} {w}_{0}^{\alpha} {(v^{\nu})}_{j+1}^{n+1} + (1 - \lambda {\delta}_{j}^{n+1} {w}_{1}^{\alpha}) {(v^{\nu})}_{j}^{n+1}
	- \lambda {\delta}_{j}^{n+1} \sum_{i = 2}^{j+1} w_{i}^{\alpha} {(v^{\nu})}_{(j+1)-i}^{n+1} 
	& = &  {(v^{\nu})}_{j}^{n} \; ,
\end{eqnarray}
where $\lambda = \frac{k c^2}{h^{\alpha}}$ and $\delta_j^n = \nu \, {({v}^{\delta})}_j^n$, to solve the Eq.~(\ref{Eq11}) with $1 < \alpha \leq 2$, $\nu > 0$, on the finite domain $L \leq x \leq R$, together with a non-negative bounded initial condition ${u}(x,0) = {u}_{0}$ and boundary conditions $u(x = L,t) = 0 = u(x = R,t)$ for all $t \geq 0$, based on the shifted Grünwald approximation as given in (\ref{Eq06}), with $p = 1$ and $h = (R-L)/M$, is consistent and unconditionally stable.
\end{theorem}
	
\begin{proof}
In this proof we adapt the argument used in the proof of Theorem $2.7$ of \cite{meerschaert}.
		
Initially, since we put a left boundary condition equal to zero, it is possible to extend $u(x,t)$ by 
$u(x,t) = 0$ for all $x < L$, $t > 0$, and the Riemann-Liouville fractional derivative (or Riesz derivative) in  (\ref{Eq11}) can be replaced by the Liouville fractional derivative of order $\alpha$ defined in Eq.~(\ref{Eq03e}). 
The Liouville fractional derivative can be approximated by the shifted Grünwald formula given in Eq.~(\ref{Eq06}). From Theorem~2.4 of \cite{meerschaert}, it follows that the order of accuracy of such approximation is $O(h)$. Therefore, the implicit method (\ref{Eq12}) is consistent with  Eq.~(\ref{Eq11}), with accuracy order $O(h)+O(k)$. 
		
Replacing in Eq.~(\ref{Eq11}) the adequate discretizations (backward difference in time-derivative and shifted Grünwald with $p = 1$ in fractional space-derivative of Riemann-Liouville or Riesz type), we obtain
\begin{eqnarray*} \label{Eq13}
	\frac{{(v^{\nu})}_{j}^{n+1} - {(v^{\nu})}_{j}^{n}}{k}
	&=&
	{\delta}_{j}^{n+1} c^2 h^{-\alpha} \sum_{i = 0}^{j+1} w_{i}^{\alpha} {(v^{\nu})}_{(j+1)-i}^{n+1} \; ,
\end{eqnarray*} where $\delta_j^n = \nu \, {({v}^{\delta})}_j^n$, with $\nu = 1 + \delta$ and $\delta > -1$. 
By considering $\lambda = k c^2 h^{-\alpha}$, we have
\begin{eqnarray} \label{Eq14}
	{(v^{\nu})}_{j}^{n+1} - {(v^{\nu})}_{j}^{n}
	&=&
	\lambda \, {\delta}_{j}^{n+1} \sum_{i = 0}^{j+1} w_{i}^{\alpha} {(v^{\nu})}_{(j+1)-i}^{n+1} \; ;
\end{eqnarray}
after a simple rearrangement of the terms in Eq.~(\ref{Eq14}) we obtain Eq.~(\ref{Eq12}), which is a linear system of the type ${A} {({v^{\alpha}})}^{n+1} = {({v^{\alpha}})}^{n}$, where ${A} = [{A}_{i,j}]$ is the matrix of coefficients that is the sum of a lower triangular with a super-diagonal matrix, described by 
\begin{eqnarray*}%\label{Eq14a}
	{A}_{0,j}
	& = & 
	\left\{
	\begin{array}{lcl}
		1, & if & j = 0 \ ; \\
		0, & if & 1 \leq j \leq M \ ,
	\end{array}
	\right. \\
	%\label{Eq14b}
	{A}_{i,j}
	& = & 
	\left\{
    \begin{array}{lcr}
		0, & \text{if} & j \geq i + 2
		\ \ \text{ for } \ 1 \leq j \leq M-1 \; ;  \\
		-\lambda \, {\delta}_{i}^{n+1} \, {w}_{0}^{\mu} & \text{if} & j = i + 1 \ , 
		\ \ \text{ for } \ 1 \leq j \leq M-1 \; ; \\
		1 - \lambda \, {\delta}_{i}^{n+1} \, {w}_{1}^{\mu} & \text{if} & j = i \ ,  
		\ \ \text{ for } \ 1 \leq j \leq M-1 \; ; \\
		-\lambda \, {\delta}_{i}^{n+1} \, {w}_{i}^{\mu} & \text{if} & j \leq i - 1 \; ,
		\ \ \text{ for } \ 1 \leq j \leq M-1 \; ,
	\end{array}
	\right. \\
	%\label{Eq14c}
	{A}_{M,j}
	& = & 
	\left\{
	\begin{array}{lcl}
		0, & if & 0 \leq j \leq M-1 \ ; \\
		1, & if & j = M \ .
	\end{array}
	\right.
\end{eqnarray*}
%$A_{0,0} = 1$ and $A_{0,j} = 0$ for $j = 1,2,\cdots,N$ and $A_{N,N} = 1$ and $A_{N,j} = 0$ for $j = 0,1,\cdots,N-1$.
%\begin{eqnarray*} \label{Eq15}
%{A}_{i,j}
% & = & 
% \left\{
%  \begin{array}{rcl}
%   0, & if & j \geq i + 2 \ ; \\
%   -\lambda \, {\delta}_{i}^{n+1} \, {w}_{0}^{\alpha}, & if & j = i + 1 \ ; \\
%   1 - \lambda \, {\delta}_{i}^{n+1} \, {w}_{1}^{\alpha}, & if & j = i \ ; \\
%   -\lambda \, {\delta}_{i}^{n+1} \, {w}_{i}^{\alpha}, & if & j \leq i - 1 \ ,
%  \end{array}
% \right.
%\end{eqnarray*}
%while $A_{0,0} = 1$, $A_{0,j} = 0$ for $j = 1,2,\cdots,N$, and $A_{N,N} = 1$, $A_{N,j} = 0$ for $j = 0,1,\cdots,N-1$.  
		
To verifying that $A$ is invertible, we initially take $z = -1$ in the Binomial formula ${(1+z)}^{\alpha} = \sum_{k=0}^{\infty} {{\alpha}\choose{k}} z^k$ and conclude that $\sum_{k = 0}^{\infty} {w}_{k}^{\alpha} = 0$. 
Since $0 \leq u_0 < \infty$, we conclude easily that $\displaystyle \max_{0\leq i \leq N}\{|\delta_i^n|\} < \infty$; therefore, it follows that $\displaystyle \sum_{k = 0}^{\infty} {\delta}_{i}^{n+1} {w}_{k}^{\alpha} = 0$ for all $i \in \{0,1,\cdots,M\}$. Moreover, since $1 < \alpha \leq 2$, the unique negative term in the sequence ${\bigl\{{w}_{k}^{\alpha}\bigr\}}_{k \in \mathbb{N}}$ is
${w}_{1}^{\alpha} = - \alpha$, and because $\delta > -1$, we have $\delta_i^n \geq 0$ for all $i \in \{0,1,\cdots,M\}$. 
Therefore,
\begin{eqnarray} \label{Eq16}
	- {\delta}_{i}^{n+1} {w}_{1}^{\alpha} = \sum_{j = 0, j \neq 1}^{\infty} {\delta}_{i}^{n+1} w_{j}^{\alpha}  
	\quad 
	\implies 
	\quad 
	- {\delta}_{i}^{n+1} {w}_{1}^{\alpha} 
	\geq  \sum_{j = 0, j \neq i}^{i+1} {\delta}_{i}^{n+1} w_{i-j+1}^{\alpha} \; .
\end{eqnarray}
This implies that
\begin{eqnarray} \label{Eq17}
	1 - {\delta}_{i}^{n+1} {w}_{1}^{\alpha} 
	\geq  \sum_{j = 0, j \neq i}^{i+1} {\delta}_{i}^{n+1} w_{i-j+1}^{\alpha} \; .
\end{eqnarray}
		
The inequality in Eq.~(\ref{Eq17}) means that the matrix $A$ is strictly  diagonal dominant by rows; hence, from Theorem~\ref{Thm1}, $A$ is invertible. 
On the other hand, let us consider $\widetilde{\lambda}$ as an eigenvalue of $A$.
If we choose $i$ such that $\displaystyle |y_i| = \max_{0 \leq j \leq M}\{|y_j|\}$, it follows that $\displaystyle \sum_{j = 0}^{M} A_{i,j} y_j = \widetilde{\lambda} y_i$, and, therefore,
\begin{eqnarray} \label{Eq18}
	\widetilde{\lambda} & = & 
	A_{i,i} + \sum_{j = 0, j \neq i}^{M} A_{i,j} \frac{y_j}{y_i} \; .
\end{eqnarray}
Now, if $i = 0$ or $i = M$, we have $\widetilde{\lambda} = 1$; otherwise, replacing the values of $A_{i,j}$ in Eq.~(\ref{Eq18}) we obtain
\begin{eqnarray*} \label{Eq19}
	\widetilde{\lambda} & = & 
	(1 - \lambda {\delta}_{i}^{n+1} {w}_{1}^{\alpha}) - \lambda {\delta}_{i}^{n+1} {w}_{0}^{\alpha} \frac{y_{i+1}}{y_i}
	- \lambda \sum_{j = 0}^{i-1} {\delta}_{i}^{n+1} w_{i-j+1}^{\alpha} \frac{y_{j}}{y_{i}} \; ,
\end{eqnarray*}
or, after a rearrangement of the terms: 
\begin{eqnarray} \label{Eq20}
	\widetilde{\lambda} & = & 
	1 - \lambda \Bigl( {\delta}_{i}^{n+1} {w}_{1}^{\alpha} + \sum_{j = 0, j \neq i}^{i+1} {\delta}_{i}^{n+1} w_{i-j+1}^{\alpha} \frac{y_{j}}{y_{i}} \Bigr) \; .
\end{eqnarray}
Since we choose $i$ such that $\bigl|\frac{y_j}{y_i}\bigr| \leq 1$, for all $0 \leq j \leq N$, by considering (\ref{Eq16}) and (\ref{Eq20}), we conclude that $\widetilde{\lambda} \geq 1$, for every $\widetilde{\lambda}$ eigenvalue of $A$. 
Thus, each  eigenvalue $\widetilde{\widetilde{\lambda}}$ of $A^{-1}$ satisfies $\widetilde{\widetilde{\lambda}} \leq 1$. Since the spectral radius of $A^{-1}$ is smaller than or equal to 1, it follows than an error $\varepsilon_{0}$ in $v^{0} \approx {u_0}$ results in an error $\varepsilon_{1}$ in $v^{1}$ less than or equal to $\varepsilon_0$ , and so on. This fact means that the error $\varepsilon_n$ in the $n$-step is bounded by the initial error $\varepsilon_{0}$, hence, the implicit method (\ref{Eq12}) is unconditionally stable. 
\end{proof}
	
\begin{corollary}\label{Cor}
The implicit Euler method (\ref{Eq12}) is convergent.
\end{corollary}

\begin{proof}
Apply Theorem~\ref{ThmR}.
\end{proof}
	
\begin{remark}
Note that one cannot apply directly the implicit Euler method (\ref{Eq12}) given in Theorem~\ref{Thm2}, because the weights ${\delta}_{j}^{n+1}$ are considered in the ($n+1$)-step, inserting, in this manner, the unknown values of the function $v$ in the ($n+1$)-step inside the coefficients of $A$. In other words, to apply the method (\ref{Eq12}), it is necessary to compute, in each time-step, the value of the weight.
%Remember that ${\delta}_{j}^{n+1} = {\nu} \, {v}^{\delta}(x_{j},t^{n+1})$, with $\delta \geq 0$.
\end{remark}
	
Next, we propose a form of iteration that allows us to implement the convergent method (\ref{Eq12}) by solve numerically Eq.~(\ref{Eq11}).
We start by evaluate the weights ${\delta}_{j}$ to the next time-step. The auxiliary scheme to be solved is: 
\begin{eqnarray} \label{Eq21}
	-\lambda {\delta}_{j}^{n} {w}_{0}^{\alpha} {(v^{\nu})}_{j+1}^{n+1} + (1 - \lambda {\delta}_{j}^{n} {w}_{1}^{\alpha}) {(v^{\nu})}_{j}^{n+1}
	- \lambda {\delta}_{j}^{n} \sum_{i = 2}^{j+1} w_{i}^{\alpha} {(v^{\nu})}_{(j+1)-i}^{n+1} 
	& = &  {(v^{\nu})}_{j}^{n} \; .
\end{eqnarray}
In the auxiliary scheme, for each time-step, all the  coefficients of $A$ are known, then $A^{-1}$ can be computed.
	
The algorithm is: 
	
\hspace{1cm} - solve the system in Eq. (\ref{Eq21});
	
\hspace{1cm} - compute ${\delta}_{j}^{n+1}$;
	
\hspace{1cm} - return a step in time and solve the system displayed in Eq.~(\ref{Eq12}). 
	
\begin{remark}
Theorem~\ref{Thm2} can be easily adapted to the equation of advection-diffusion Eq.~(\ref{Eq02}): it is sufficient to discretize implicitly the advective term utilizing the backward difference operator, and then evaluate the font term $f(x,t)$ in the grid points. The discretization of the advective term produces a new positive term in the coefficients $A_{i,i}$, and also yields a new negative term in the coefficients $A_{i,i-1}$ (both without weights $\delta_{i}^{n+1}$) keeping, therefore, the matrix $A$ invertible. 
\end{remark}
	
\begin{remark}
In the case of $\alpha = 2$, the shifted Grünwald–Letnikov derivative formula~(\ref{ShiG}) agrees with the second-order centered difference operator. Therefore, the method presented in this note also holds for equations (\ref{Eq01}) and (\ref{Eq02}) with (the usual) integer order derivative.
\end{remark}
	
\section{Example of Application}
	
Several diffusive phenomena in nature are satisfactory modeled by the Fokker-Plank linear equation
\begin{eqnarray} \label{Eq22}
	\frac{\partial }{\partial t} P(x,t) = D \frac{{\partial}^{2}}{\partial x^2} P(x,t)  \; ,
\end{eqnarray}
where $P(x,t)$ is the density of probability in the $x$-space, $D > 0$ is the diffusion coefficient. One interesting point about the processes described by Eq.~(\ref{Eq22}) concerns the Markovian characteristics, which is underlined by the linear time dependence manifested by the mean square displacement, i.e., $\left \langle \left(x-\langle x \rangle \right)^2 \right \rangle \propto t$ (usual diffusion). However, many situations (see for example Refs.~\cite{Book2018,Tsallis,10.1080,C4CP03465A,book12472}) have shown a different behavior (e.g., anomalous diffusion) of those modeled by Eq.~(\ref{Eq22}). In order to face these scenarios, Eq.~(\ref{Eq22}) has been extended by incorporating fractional space-derivatives and nonlinear terms. It is worth mentioning that the space-fractional derivatives has been related to the L\'evy distributions and the nonlinear case to correlated-like diffusive processes (see for example Refs.\cite{Book2018,Tsallis}). In Ref. \cite{Tsallis} it is considered the one-dimensional equation 
\begin{eqnarray} \label{Eq23}
	\frac{\partial }{\partial t} P(x,t) = D \frac{{\partial}^{\gamma}}{\partial {x}^{\gamma}} {P}^{\nu}(x,t)  \; , 
	\quad -\infty < \gamma \leq 2, \, \quad \nu > -1, \, \quad x \in \mathbb{R}, \quad t > 0,
\end{eqnarray}
and, under some hypotheses, exact time-dependent solutions are exhibited to $\gamma$ in several subintervals of $\left(-\infty, 2\right]$.
	
In the following, we consider some important cases of Eq. (\ref{Eq23}) in a bounded domain (by simplicity, we put $D = 1$) and we utilize our algorithm to present some numerical solutions for the nonlinear initial-boundary value problem 
\begin{eqnarray} \label{Eq24}
	\left\{
	\begin{array}{l}
		\displaystyle \frac{\partial }{\partial t} u(x,t) = \frac{{\partial}^{\alpha}}{\partial {x}^{\alpha}} {u}^{\nu}(x,t)  \; , 
		\quad 1 < \alpha \leq 2, 
		\quad \nu > 0, 
		\quad 0 \leq x \leq 5, \, \\
		\displaystyle u(0,t) = 0 = u(5,t) \; , \quad \forall \ \ t > 0, \\
		\displaystyle 
		u(x,0) = u_0(x) \; , \quad \forall \ \ 0 \leq x \leq 5, 
	\end{array}
	\right.
\end{eqnarray}
where the initial data $u_0$ is the Gaussian function given by
\begin{eqnarray*} \label{Eq25}
	{u}_{0}(x) = \frac{2}{\sqrt{2 \pi {(0.3)}^{2}}} \, 
	\exp\left(\frac{-{(x - 2.5)}^{2}}{2 {(0.3)}^{2}}\right) \; .
\end{eqnarray*}
	
In Figure~\ref{fig2.3}, we present some numerical solutions to Problem (\ref{Eq24}) at $t = 1s$: in Figure~\ref{fig2.3.a}, we can see some linear cases ($\nu = 1$), and the effects to fractional space-derivative when $\alpha = 1.2$, $\alpha = 1.4$, $\alpha = 1.6$, $\alpha = 1.8$ and $\alpha = 2.0$ (the normal diffusion case displayed in Table 1). In Figure~\ref{fig2.3.b}, some nonlinear cases of Problem~(\ref{Eq24}) are presented, where $\alpha = 1.5$ and $\nu$ takes the values $\nu = 2.0$, $\nu = 1.5$, $\nu = 1.0$ (linear case), $\nu = 0.5$ and $\nu = 0.2$.
	
Since we want to compare our numerical solution with the exact solution of the Problem (\ref{Eq24}), we choose the well-known second-order linear equation as Eq.~(\ref{Eq22}). In Table~1, the exact values of the solution $u$ are obtained from the variable separation method; the numerical values $v$ are obtained from our algorithm, and the error considered is $E = v - u$. The time considered is $t = 1s$. 
	
%%%%%%%%%%% FIGURES %%%%%%%%%%%%
\begin{figure}[!htb]
\centering
    \subfloat[Numerical solutions for $\nu = 1.0$, \newline with $\alpha = 1.2$, $\alpha = 1.4$, $\alpha = 1.6$, $\alpha = 1.8$ and $\alpha = 2.0$.]{
		\includegraphics[height=5.2cm]{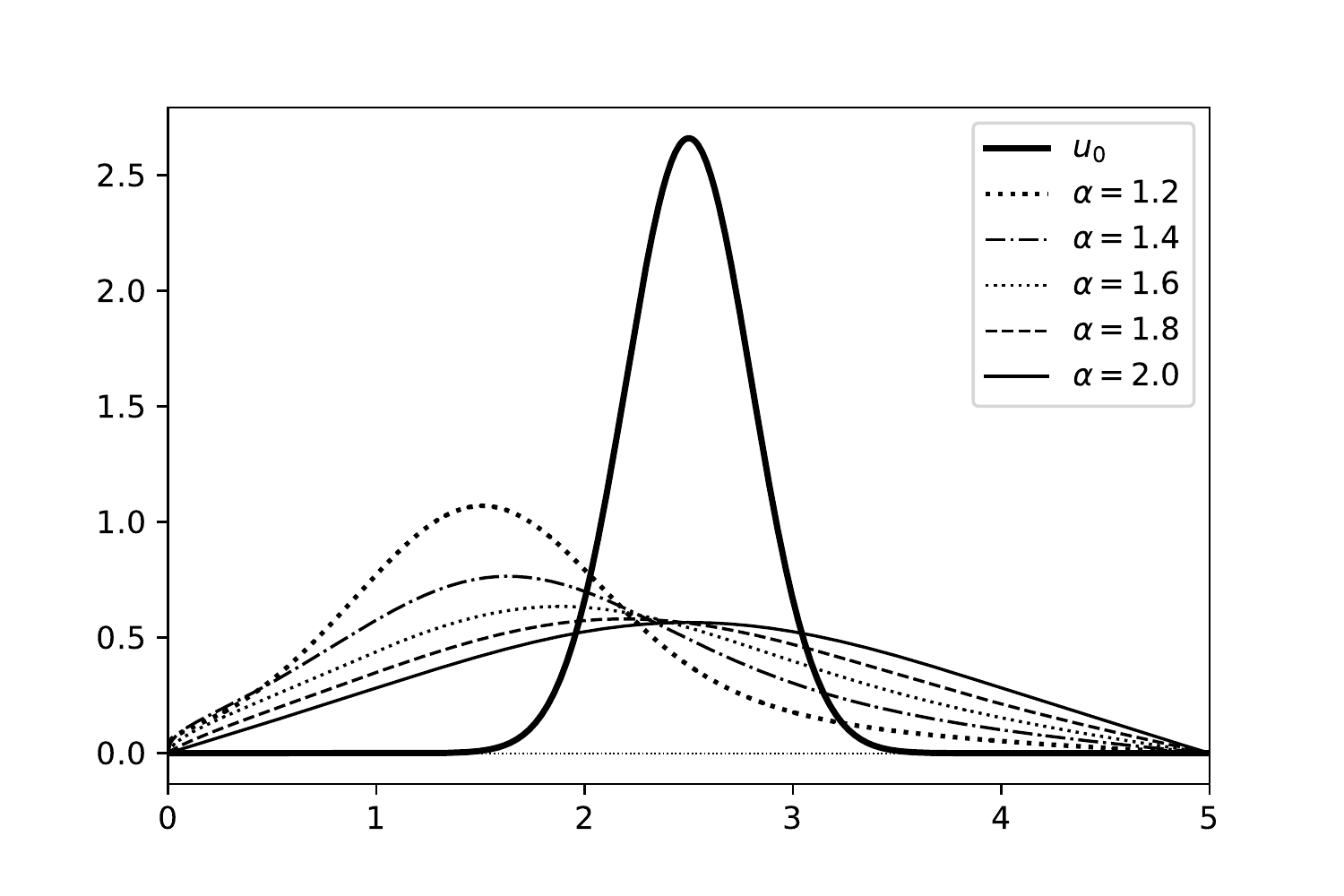}
		\label{fig2.3.a}
        }
    \ \ \ \ %espaco separador
    \subfloat[Numerical solutions for $\alpha = 1.5$, \newline with $\nu = 0.2$, $\nu = 0.5$, $\nu = 1.0$, $\nu = 1.5$ and $\nu = 2.0$.]{
		\includegraphics[height=5.2cm]{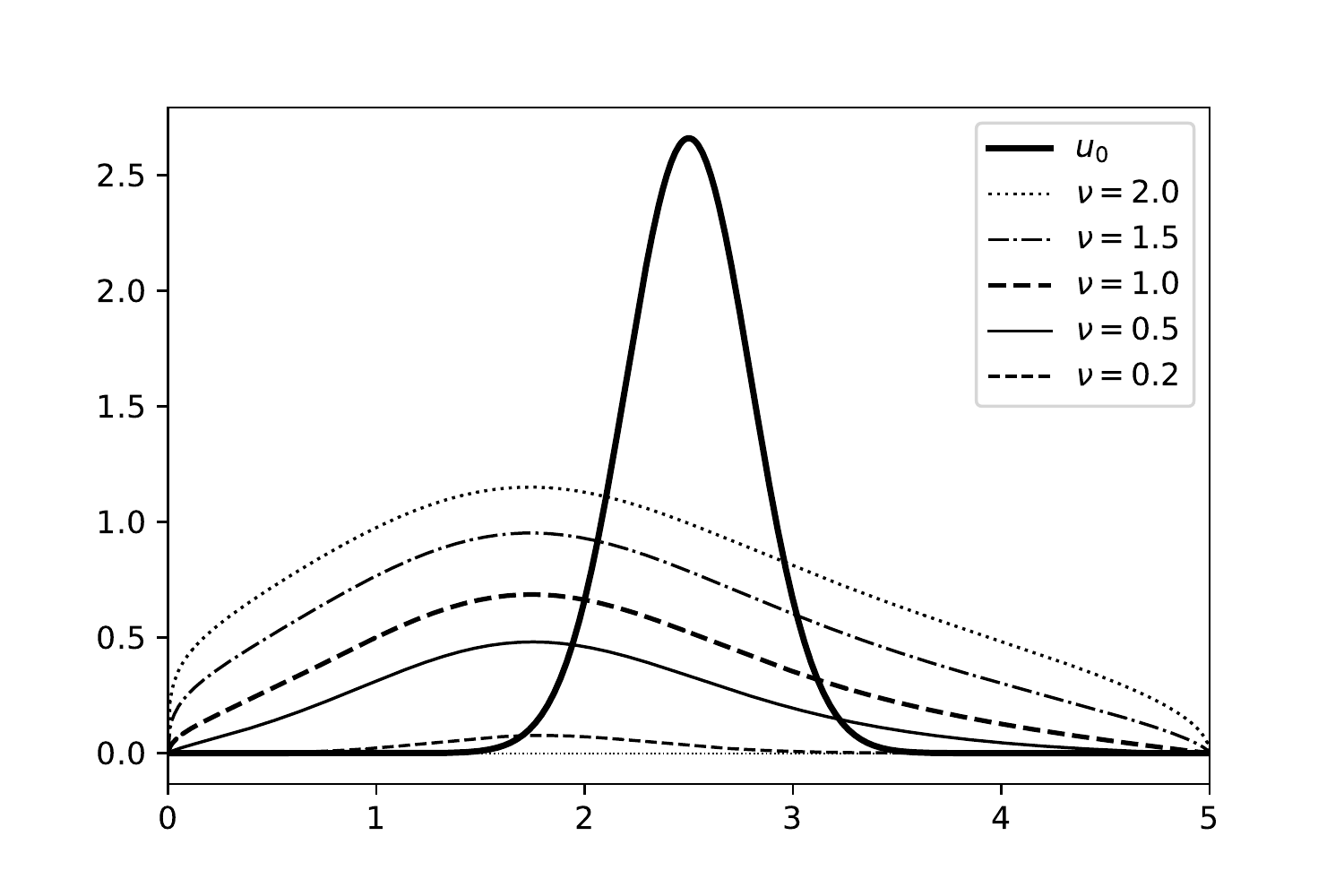}
		\label{fig2.3.b}
		}
	\caption{Numerical solutions at $t = 1$ for $\nu = 1.0$ and $\alpha$ variable, and for $\alpha = 1.5$ and $\nu$ variable.}
	\label{fig2.3}
\end{figure}
	
\begin{remark}
In Ref.~\cite{Tsallis}, in order to obtain the exact solution of Eq.~(\ref{Eq23}), the authors assume that $\nu = \frac{2 - \gamma}{1 + \gamma}$. In this sense, the unique linear case considered in Ref.~\cite{Tsallis} is when $\nu = 1$ and $\gamma = 0.5$ in Eq.~(\ref{Eq23}). Unfortunately,  such $\gamma$ is not compatible with Theorem~\ref{Thm2} proposed here, because $1 < \alpha \leq 2$. This is the reason why we consider Eq.~(\ref{Eq22}) to perform our comparison.
\end{remark}
	
\begin{table}[h]\label{tabela1}
	\center
	\small{ {\bf Table 1} - Values of numerical solution $v$, exact solution $u$, and error $E = v - u$, \\ for $\alpha = 2$ and $\nu = 1$ in Problem (\ref{Eq24}), at $t=1$.} \\
	\begin{tabular}{c|c|c|c}
		\hline \hline
		$x$   &    $v$    &     $u$      &   $E$    \\ \hline \hline
		0,0 \ \ & \ \ 0,00000000 \ \ & \ \ 0,00000000 \ \ & \ \ \ \ 0,00000000 \\ \hline
		0,5 \ \ & \ \ 0,13650936 \ \ & \ \ 0,14492708 \ \ & \ \ - 0,00841772 \\ \hline
		1,0 \ \ & \ \ 0,28017859 \ \ & \ \ 0,28987258 \ \ & \ \ - 0,00969399 \\ \hline
		1,5 \ \ & \ \ 0,41875250 \ \ & \ \ 0,41990747 \ \ & \ \ - 0,00115497 \\ \hline
		2,0 \ \ & \ \ 0,52474526 \ \ & \ \ 0,51329497 \ \ & \ \ \ \ 0,01145029 \\ \hline
		2,5 \ \ & \ \ 0,56558688 \ \ & \ \ 0,54801760 \ \ & \ \ \ \ 0,01756928 \\ \hline
		3,0 \ \ & \ \ 0,52629428 \ \ & \ \ 0,51462415 \ \ & \ \ \ \ 0,01167013 \\ \hline
		3,5 \ \ & \ \ 0,42130602 \ \ & \ \ 0,42221621 \ \ & \ \ \ \ 0,00208981 \\ \hline
		4,0 \ \ & \ \ 0,28306465 \ \ & \ \ 0,29267907 \ \ & \ \ - 0,00961442 \\ \hline
		4,5 \ \ & \ \ 0,13934058 \ \ & \ \ 0,14787562 \ \ & \ \ - 0,00853504 \\ \hline
		5,0 \ \ & \ \ 0,00000000 \ \ & \ \ 0,00000000 \ \ & \ \ \ \ 0,00273680 \\ \hline \hline
	\end{tabular}
	%\caption{Valores da solução exata $u$, da solução numérica $v$, e do erro $E$, para o tempo $t=1$} \label{tabela1}
\end{table}
	
Let us now analyze the CPU time and the conditions in which our method was applied. All computations were made in Phyton language, utilizing the Spyder software. The simulations presented in the Figure~\ref{fig2.3} were obtained with $ M = 500$ ($h = 0.01$) and $N = 10$ ($k = 0.1$), i.e., the desired solution at time $1$s was generated after $10$ time-steps iterations. In these conditions, the CPU time of our nonlinear method (\ref{Eq12}) was between 209 and 212 seconds for each simulation. Although we did not get any nonlinear method similar to our method to perform the comparison, we then compare our method with a linear one. For the cases in the Figure~\ref{fig2.3.a}, if the linear method based on \cite{meerschaert} is used, with the same conditions, the CPU time was around 20 seconds. It is interesting to note that in each time step of our nonlinear method, the algorithm inverts an $M \times M$ matrix (in our case, $M=500$), whereas in the linear method, a unique $M \times M$ matrix is necessary to invert. In this light, the CPU time of our nonlinear method is reasonable; this is the price to pay by the nonlinearity of the system.

\section{Final Remarks}
	
We have proposed a convergent numerical method based on finite differences to solve a class of nonlinear advection-diffusion fractional differential equation, which are utilized to model, for instance, the porous media as well as phenomena which present anomalous diffusion.  
%As future works, it will be interesting to investigate how to solve numerically FPDEs for $\nu \leq 0$ and $0 < \alpha < 1$. 
We hope that the results presented here be useful to discuss/solve nonlinear advection-diffusion fractional differential equations in connection with the anomalous diffusion.
	
\section*{Acknowledgment}
This research has been partially supported by the Brazilian Agencies CAPES and CNPq.
	
\bibliographystyle{unsrt}  
\bibliography{references}

\end{document}